\documentclass{ijmart}

\usepackage{amsmath}
\usepackage{amssymb}
\usepackage{amsthm}

\usepackage{times}

\DeclareMathOperator{\supp}{supp}
\DeclareMathOperator{\IP}{IP}

\newcommand{\N}{\mathbb N}

\mathchardef\mhyphen="2D

\newtheorem{lemma}{Lemma}[section]
\newtheorem{theorem}{Theorem}[section]
\newtheorem{proposition}{Proposition}[section]
\newtheorem{definition}{Definition}[section]
\newtheorem{question}{Question}[section]
\newtheorem{example}{Example}[section]

\title{Transfinite Approximations to Hindman's Theorem}
\author[Mathias Beiglb\"ock and Henry Towsner]{Mathias Beiglb\"ock and Henry Towsner\thanks{The first author acknowledges financial support from the Austrian Science Fund (FWF) under grant P21209.}}

\date{\today}

\begin{document}
\maketitle

\begin{abstract}
Hindman's Theorem states that in any finite coloring of the integers, there is an infinite set all of whose finite sums belong to the same color. This is much stronger than the finite analog stating that for any $n,r$, there is a $k$ such that for any $r$-coloring of $[1,k]$, there is a set of $n$ integers all of whose finite sums belong to the same color.

We extend the finite form of Hindman's Theorem to an $\alpha$-Hindman Theorem for each countable ordinal $\alpha$. These $\alpha$-statements approximate Hindman's Theorem in the sense that the full fledged theorem is equivalent to the transfinite version holding for every countable ordinal. 
 
We then give a proof of Hindman's Theorem by directly proving these transfinite approximations.
\end{abstract}

\section{Introduction}

It will be convenient for us to take $\N$ to be the set of positive integers.  (That is, $0\not\in \N$.)  Also, ``integer'' means positive integer throughout.

\begin{definition}
  If $\sigma$ is a (finite or infinite) set of integers, we write
\[FS(\sigma):=\big\{n_1+ \ldots + n_k \mid   \{n_1< \ldots< n_k\}\subseteq\sigma\big\},\] the set of non-empty finite sums from $\sigma$. $A\subseteq \N$ is an $\IP$-set if there exists  there is an infinite set $\tau$ with $FS(\tau)\subseteq A$, and it is an $\IP_n$-set if there exists such a $\tau$ with $|\tau|=n$.
\end{definition}

We recall  Hindman's Theorem (\cite{Hind74}) and its finitary version, Folkman's Theorem.\footnote{The finitary version of Hindman's Theorem is baptized after Folkman in \cite{GrRS80} and we follow this tradition. It is a special case of Rado's Theorem \cite{Rado33} and was proved independently (but much later than Rado's Theorem) by  Folkman (unpublished) and Sanders \cite{Sand69}.}

\begin{theorem}[Hindman's Theorem]\label{HindmanTheorem}
For any finite coloring $c:\mathbb N\to [1,r]$ of the positive integers there exists a monochromatic $\IP$-set. 
\end{theorem}

\begin{theorem}[Folkman's Theorem]\label{FolkmanTheorem}
For any finite coloring $c:\mathbb N\to [1,r]$ of the positive integers and for every $n>0$ there exists a monochromatic $\IP_n$-set. 
\end{theorem}

We establish Hindman's Theorem by proving a sequence of statements, indexed with the ordinals less than $\omega_1$ whose strength lies strictly below Hindman's Theorem.

The key objects are $\IP_\alpha$-sets (where $\alpha<\omega_1$) which ``interpolate'' between finite $\IP_n$-sets and full fledged $\IP$-sets.  (This idea was inspired by \cite{Hir04}.)  To motivate the inductive definition, we note that a set $A$ is an $\IP_{n+1}$-set if and only if there exists an integer $a\in A$ such that $A\cap (A-a)$ is an $\IP_n$-set.

\begin{definition}
  Every nonempty set  (of positive integers) is $\IP_1$. If there is an $a\in A$ such that $A\cap (A-a)$ is $\IP_\alpha$ then $A$ is $\IP_{\alpha+1}$.  If for every $\beta<\lambda$ $A$ is $\IP_\beta$ then $A$ is $\IP_\lambda$.
\end{definition}  

An $\IP_\omega$-set simply contains $\IP_n$-sets for every $n\in \N$, but unlike in $\IP$-sets, there is no requirement that those finite sets are mutually related. More exotic are the $\IP_{\omega+1}$-sets which contain a single $a$ and for each $n$ all of $FS(a,b_1^n,\ldots, b_n^n)$ for some $b_1^n,\ldots, b_n^n$. Next, an $\IP_{\omega+k}$ set $A$ has a core $\{a_1,\ldots,a_k\}$ such that for each $n$ there exist $b_1^n,\ldots, b_n^n$  with $FS(a_1, \ldots, a_k,b_1^n,\ldots, b_n^n)\subseteq A.$ Similarly we find at step $\omega+\omega$, that for each $k$ there is size $k$ core and for each $n$ a size $n$ extension, all of whose finite sums are contained in the $\IP_{\omega+\omega}$-set. 

We depart from the increasingly difficult task of describing $\IP_\alpha$-sets and clarify the role of $\IP$-sets with respect to this hierarchy.
\begin{proposition}\label{IPRole}
A set $A\subseteq \N$ is an $\IP$-set iff it is an $\IP_{\omega_1}$-set.
\end{proposition}
\begin{proof}
  First, note that if $A$ is $\IP_{\omega_1}$ then there is some $a_0\in A$ such that $A\cap (A-a_0)$ is $\IP_{\omega_1}$.  For if not, for every $a\in A$, there would be a $\gamma_a$ such that $A\cap (A-a)$ is not $\IP_{\gamma_a}$, and therefore we could take the supremum of these ordinals to find a $\gamma$ such that $A$ was not $\IP_{\gamma+1}$.
  
  So suppose $A$ is $\IP_{\omega_1}$.  Then we may choose $a_0\in A$ such that $A\cap (A-a_0)$ is $\IP_{\omega_1}$, and then $a_1\in A\cap (A-a_0)$ such that $A\cap (A-a_0)\cap (A-a_1)\cap (A-a_0-a_1)$ is $\IP_{\omega_1}$, and so on.  Iterating, we obtain a sequence $\{a_i\}$ such that $FS(\{a_i\})\subseteq A$.

Conversely, suppose $A$ were not $\IP_{\gamma_0}$ for some $\gamma_0$, but that $\{a_i\}$ were a sequence of integers such that $FS(\{a_i\})\subseteq A$.  Then $A\cap (A-a_0)$ is not $\IP_{\gamma_1}$ for some $\gamma_1<\gamma_0$, and $A\cap (A-a_0)\cap (A-a_1)\cap (A-a_0-a_1)$ is not $\IP_{\gamma_2}$ for some $\gamma_2<\gamma_1$.  Iterating, we obtain an infinite descending sequence of ordinals, which is a contradiction.
\end{proof}

A version of Folkman's Theorem asserts that for all $n$ and $r$ there is some $m$, such that whenever an $\IP_m$-set is $r$-colored, there exists a monochromatic $\IP_n$-set. (See the comment following Lemma \ref{UnmeshedFolkman}.) Remarkably, it is possible to prove an ordinal analogue of this statement which yields Hindman's Theorem. 
\begin{theorem}\label{AlphaFolkman} 
For any $\alpha<\omega_1$, there exists some $\beta <\omega_1$, such that for every  finite coloring of an $\IP_{\beta}$-set  there exists a monochromatic $\IP_{\alpha}$-set.  
\end{theorem}
To see that Theorem \ref{AlphaFolkman} yields Hindman's Theorem, fix a finite coloring of $\N$. By Proposition \ref{IPRole}, $\N$ is $\IP_{\omega_1}$. Thus, by Theorem \ref{AlphaFolkman}, there exists a monochromatic $\IP_\alpha$-set for each $\alpha<\omega_1$. By the pigeonhole principle, there exists a color class which is $\IP_\alpha$ for every $\alpha< \omega_1$, whence $\IP_{\omega_1}$, whence $\IP$.

Below we shall prove that the converse is true as well, i.e.\ a short argument makes it possible to derive Theorem \ref{AlphaFolkman} from Hindman's Theorem.   

Moreover we shall prove a version of Theorem \ref{AlphaFolkman} (Theorem \ref{OrdinalEffectiveHindman}) which is effective in the sense that $\beta=\beta(\alpha)$ is calculated explicitly from $\alpha$, thus providing a new proof of Hindman's Theorem.

While the function obtained in Theorem \ref{OrdinalEffectiveHindman} grows rapidly, there exist ``good bounds'' for the first infinite ordinals. By Folkman's Theorem, every finite coloring of an $\IP_\omega$-set admits a monochromatic $\IP_\omega$-set. The same behavior occurs at level $\omega^2$. 

\begin{theorem}\label{SquaredFolkman}
For any finite coloring of an $\IP_{\omega^2}$-set  there exists a monochromatic $\IP_{\omega^2}$-set.  
\end{theorem}

The proof of Theorem \ref{SquaredFolkman} turns out to be rather simple from a conceptual point of view; indeed it does not require more than repeated applications of Folkman's Theorem and the pigeonhole principle.\footnote{It might be interesting to note that
while Theorem \ref{SquaredFolkman} 
doesn't require new combinatorial arguments, it extends Folkman's Theorem even on a completely finitistic level. For instance it implies a ``Paris-Harrington''-type Folkman's Theorem.} 

Unfortunately we do not know if a similar technique is applicable for higher ordinals. A straightforward generalization would suggest that $\IP_{\omega^3}$-sets are partition regular as well, but this fails badly. Instead is possible to two-color an $\IP_{\omega^3}$-set, such that there is no monochrome $\IP_{\omega^2+1}$-set (see Example \ref{NoCubedFolkman}). 

The paper is organized as follows. Section \ref{SquaredFolkmanSection} is a warm up section in which we give the proof of Theorem \ref{SquaredFolkman}.  In Section \ref{TreeDefinitionSection} we describe $\IP_\alpha$-sets in terms of trees and prove that Hindman's Theorem is equivalent to a claim about $\IP_\alpha$-sets.  Finally, in Section \ref{HindmansTheoremSection}, we give an explicit proof of Theorem \ref{AlphaFolkman}.

\section{$\IP_{\omega^2}$-sets are partition regular}\label{SquaredFolkmanSection} 

We will need the following somewhat strengthened version of Folkman's Theorem:
\begin{lemma}\label{UnmeshedFolkman}
For all $r,n\in\N$ there is some $m=m(r,n)$ such that for any $r$-coloring $c$ of $\mathcal{P}([1,m])\setminus\{\emptyset\}$ there exist sets $\tau_1,\ldots, \tau_n$ and $\tilde r\in [1,r]$ such that $\max \tau_i< \min \tau_{i+1}, i\in [1,n-1]$ and $c$ takes the value $\tilde r$ on all sets of the form $\tau_{i_1}\cup \ldots \cup\tau_{i_k},$ where $1\leq i_1< \ldots< i_k\leq n.$
\end{lemma}
\begin{proof}
Fix $r, n\in\N$. Pick, using Folkman's Theorem (and a standard compactness argument), an integer $l$ such that for any $r$-coloring of $[1,l]$ there exist $a_1,\ldots, a_n$ such that $FS(\{a_1,\ldots, a_n\})$ is monochromatic. 

Pick, using Ramsey's Theorem, an integer $m$ such that for any $r$-coloring of ${\mathcal P}(X)$, where $|X|=m$, there exists a set $Y\subseteq X, |Y|=l$ with the color of the subsets of $Y$ depending solely on their cardinality.

Given an $r$-coloring $c$ of the finite non-empty subsets of $[1,m]$, pick $Y=\{y_1< \ldots<y_{l}\}\subseteq [1,m]$ such that for each $k\in [1,l]$ there exists a color $c'(k) $ with $c(\tau)= c'(|\tau|)$ for $\tau\subseteq Y$. By the choice of $l$, pick $a_1, \ldots,a_n$ such that $c'$ is constant on $FS(\{a_1,\ldots, a_n\})$. Finally set $\tau_1=\{y_1, \ldots, y_{a_1}\}, \ldots, \tau_k=\{y_{a_1+\ldots+a_{n-1}+1},\ldots, y_{a_1+\ldots+a_{n}}\}$.  Then whenever $1\leq i_1<\cdots<i_k\leq n$, $c(\tau_{i_1}\cup\cdots\cup\tau_{i_k})=c'(|\tau_{i_1}\cup\cdots\cup\tau_{i_k}|)=c'(a_1+\cdots+a_k)$, which is constant.
\end{proof}


The crucial point of the proof of Theorem \ref{SquaredFolkman} is the observation that every coloring of an $\IP_{\omega\cdot m}$-set $A$ behaves rather simply on a properly chosen $\IP_{\omega\cdot m}$-subset of $A$.

Given $l_1, \ldots, l_m$ we set  $D(l_1, \ldots,l_m )= \mathcal{P}([1,l_1])\times \ldots \times \mathcal{P}([1,l_m]) \setminus\{(\emptyset,\ldots,\emptyset)\}$ and $\supp g =\{k\in [1,m]\mid g(k)\neq \emptyset\}$ for $g\in D(l_1, \ldots,l_m )$.
\begin{lemma}\label{SimpleColoring}
Let $c$ be a finite coloring of an $\IP_{\omega\cdot m}$-set $A$. Then there exists a finite coloring $\tilde c$ of $\mathcal{P}([1,m])\setminus\{\emptyset\}$ such that
\begin{align*}
\forall l_1\exists a_1^1,\ldots, a_{l_1}^1&\\
\forall l_2\exists a_1^2,\ldots,a_{l_2}^2 \\
\cdots\\
\forall l_m\exists a_1^m,\ldots, a_{l_m}^m 
&FS(a^1_1, \ldots, a^m_{l_m})\subseteq A\\ 
&\mbox{and}\  \forall g\in D(l_1,\ldots,l_m)\ c\big(\textstyle{\sum_{i=1}^m\sum_{j\in g(i)}a^i_{j})}\big)=\tilde c (\supp g).
\end{align*}
\end{lemma}
\begin{proof}
By induction. The case $m=1$ is Folkman's Theorem. 
Suppose the claim holds for $m$ and let $c$ be an $r$-coloring of an $\IP_{\omega\cdot (m+1)}$-set $A$.  Let $l_0$ be given, and let $n$ be large enough that Folkman's Theorem guarantees a monochromatic $\IP_{l_0}$-subset for any $ r^{2^m}$-coloring of an $\IP_n$-set.

Since  $A$ is $\IP_{\omega\cdot m+n}$, we obtain a sequence $a^0_1,\ldots,a^0_n$ such that $FS(a^0_1,\ldots,a^0_n)\subseteq A$ and $A'=\bigcap_{a\in FS(a^0_1,\ldots,a^0_n)\cup\{0\}}A-a$ is $\IP_{\omega\cdot m}$.

Applying the inductive hypothesis to the coloring $$b\mapsto (c(a+b))_{a\in FS(a^0_1,\ldots,a^0_n)\cup\{0\}}$$ on $A'$ we find that there exist colorings $\tilde c_a$, $a\in FS(a^0_1,\ldots,a^0_n)\cup\{0\}$ on $\mathcal{P}([1,m])\setminus\{\emptyset\}$ such that
\begin{align*}
\forall l_1\exists a_1^1,\ldots, a_{l_1}^1&
\forall l_2\exists a_1^2,\ldots,a_{l_2}^2 
\cdots
\forall l_m\exists a_1^m,\ldots, a_{l_m}^m,  FS(a^0_1, \ldots, a^m_{l_m})\subseteq A, \\
\mbox{and }& \forall g\in D(l_1,\ldots,l_m)\forall  a\in FS(a^0_1,\ldots,a^0_n)\cup\{0\}  \\
& c\big(a+\textstyle{\sum_{i=1}^m\sum_{j\in g(i)}a^i_{j})}\big)=\tilde c_a (\supp g).
\end{align*}
Then the map $c'(a)=(c(a),\tilde c_a)$ is an $r\cdot r^{ 2^m-1}$-coloring of $FS(a_1^0,\ldots,a_n^0)$.  By choice of $n$, there are $a'_1,\ldots,a'_{l_0}$ with $FS(a'_1,\ldots,a'_{l_0})\subseteq FS(a_1^0,\ldots,a_n^0)\subseteq A$ such that $FS(a'_1,\ldots,a'_{l_0})$ is monochromatic under $c'$.  Then, taking $a$ to be any element of $FS(a'_1,\ldots,a'_{l_0})$, define an $r$-coloring of $\mathcal{P}([0,m])\setminus\{\emptyset\}$ by
\begin{align*}
\tilde c_{l_0}(S)=\left\{\begin{array}{ll}
\tilde c_0(S)&\text{if }0\not\in S\\
\tilde c_a(S\setminus \{0\})&\text{if }0\in S\text{ and }S\neq\{0\}\\
c(a)&\text{if }S=\{0\}
\end{array}\right..
\end{align*}
Finally pick an $r$-coloring $\tilde c$ of $\mathcal{P}([0,m])\setminus\{\emptyset\}$ so that $\tilde c=\tilde c_{l_0}$ for infinitely many $l_0$.
\end{proof}

\begin{proof}[Proof of Theorem \ref{SquaredFolkman}.]
Fix $r$ and $n$ and pick $m=m(r,n)$ as in Theorem \ref{UnmeshedFolkman}. Let $c$ be an  $r$-coloring of an $\IP_{\omega\cdot m} $-set $A$. 
Pick $\tilde c$ according to Lemma \ref{SimpleColoring}. Pick $\tau_1, \ldots, \tau_n$ and a color $\tilde r\in [1,r]$ according to Lemma \ref{UnmeshedFolkman} applied to the coloring $\tilde c$.

It suffices to show that
\begin{align*}
  \forall L_1\exists b_1^1,\ldots, b_{L_1}^1&\\
  \forall L_2\exists b_1^2,\ldots,b_{L_2}^2 \\
  \cdots\\
  \forall L_n\exists b_1^n,\ldots, b_{L_n}^n
  &FS(b^1_1, \ldots, b^n_{L_n})\subseteq A\\
&\mbox{and}\ \forall b\in FS(b^1_1, \ldots, b^n_{L_n})\ c(b)=\tilde r,
\end{align*}
since the set of all such $b$ gives rise to an $\IP_{\omega\cdot n}$-subset of $A$.

Given $L_1,\ldots,L_n$, we will describe the construction of these sequences, taking care that the choice of $b_1^j,\ldots,b_{L_j}^j$ depends only on $L_i$ with $i\leq j$.  For each $j\leq m$, let $l_j=L_k$ if $j\in \tau_k$; if $j\not\in\tau_k$ for any $k$, let $l_j$ be arbitrary (say, $1$).

For each $k\leq n, i\leq L_k$, we let $b_i^k=\sum_{j\in\tau_k}a_i^j$.  Note that the choice of $b_i^k$ depends only on $a_i^j$ for $j\leq\max\tau_k$, which in turn depends only on $L_{k'}$ for $k'\leq k$, so we have satisfied the requirements given by the order of the quantifiers.

Each $b\in FS(b^1_1, \ldots, b^n_{L_n})$, has the form $b=\sum_{i=1}^m\sum_{j\in g(i)}a^i_{j}$ for some $g$ with $\supp(g)=\tau_{i_1}\cup\cdots\cup\tau_{i_k}$ for some $i_1<\cdots<i_k$.  
Therefore $c(b)=\tilde c(\tau_{i_1}\cup\cdots\cup\tau_{i_k})=\tilde r$, as desired.
\end{proof}

\section{Trees of Integers}\label{TreeDefinitionSection}
As the descriptions above indicate, it becomes unwieldy to describe $\IP_\alpha$-sets explicitly as $\alpha$ gets large.  Instead, we introduce the following notion of a tree.  (This definition differs slightly from usual ones, to simplify the statements and proofs below.)

\begin{definition}

  A \emph{tree} $T$ is a collection of non-empty finite sets of integers such that if $\sigma\in T$ and $\emptyset\neq\tau\subseteq\sigma$ then $\tau\in T$.  We write $\sigma<n$ if $m<n$ for all $m\in FS(\sigma)$.  We write $\sigma<\tau$ if $\tau$ is nonempty and $\sigma<n$ for every $n\in FS(\tau)$.  We write $\sigma<T$ if $\sigma<\tau$ for every $\tau\in T$.

We write $FS(T)$ for $\bigcup_{\sigma\in T}FS(\sigma)$.  If $T,T'$ are trees, we write $T'\subseteq_{FS}T$ if $FS(T')\subseteq FS(T)$.

If $\sigma\in T$, we define a new tree, $T-\sigma$, by setting $\tau\in (T-\sigma)$ iff $\sigma<\tau$ and $\sigma\cup\tau\in T$.  When $\sigma<T$, $T+\sigma$ is the tree generated by $T\cup\{\sigma\cup\tau\mid \tau\in T\}$.

We may view $T$ as being partially ordered by the relation $\sigma\prec\tau$ iff $\tau\subseteq\sigma$ and there is a $\sigma'\subseteq\sigma$ such that $\tau<\sigma'$.  We write $ht(T)$ for the ordinal height of $T$ under this ordering if it is well-founded, and $ht(T)=\infty$ if $T$ is ill-founded.

\end{definition}
For every countable ordinal $\alpha$, there is a tree with $ht(T)=\alpha$; for instance, letting $\pi:\alpha\rightarrow\N$ and $p:\N\times\N\rightarrow\N$ be bijections, and we can define $T$ by setting $\sigma\in T$ iff $\sigma$ has the form $\{p(1,\pi(\gamma_1)),\ldots,p(k,\pi(\gamma_k))\}$ where $\alpha>\gamma_1>\gamma_2>\cdots>\gamma_k$.

Note that saying $T$ is ill-founded under $\prec$ means that it contains an infinite path, which is an infinite set $\Lambda$ with every finite subset of $\Lambda$ belonging to $T$.  Moreover, note that any infinite set with all finite subsets belonging to $T$ gives an infinite $\prec$-descending sequence (since given $\tau\subseteq\Lambda$, we may always find a $\sigma'\subseteq\Lambda$ with $\tau<\sigma'$).  Also, since the cardinality of $T$ is countable, either $ht(T)<\omega_1$ or $ht(T)=\infty$. 

Note also that $ht(T)\geq \alpha+\beta$ implies $ht\{\sigma:ht(T-\sigma)\geq \alpha\}\geq \beta$.

The following is the first of many proofs by transfinite induction in this paper.  In such proofs, we will write ``IH'' to abbreviate the inductive hypothesis.

\begin{proposition}
  If there is a tree $T$ of height $\alpha$ with $FS(T)\subseteq A$ then $A$ is an $\IP_\alpha$-set.
\end{proposition}
\begin{proof}
  By induction on $\alpha$.  When $\alpha=0$, this is trivial, and the limit stage is trivial as well.  If there is a tree $T$ of height $\alpha+1$ with $FS(T)\subseteq A$, choose $\{a\}\in T$ with $ht(T-\{a\})\geq\alpha$.  Since $FS(T-\{a\})\subseteq A$ and $FS(T-\{a\})\subseteq A-a$, it follows that $A\cap (A-a)$ is an $\IP_\alpha$ set, and therefore $A$ is an $\IP_{\alpha+1}$-set.
\end{proof}

The reverse is not quite true: to obtain an exact equivalence, we would want to consider the height of trees under $\supseteq$ rather than $\prec$.  However it will be more convenient to work with height under $\prec$. And the reverse is still almost true in the sense that it holds for limit ordinals.

\begin{proposition}
If $A$ is an $\IP_\lambda$-set for a limit ordinal $\lambda$, then there is a tree $T$ of height $\lambda$ such that $FS(T)\subseteq A$.
\end{proposition}
\begin{proof}
For a tree $T$ and $n\in\N$, write $T^{(n)}:=\{\tau\cap\{n,n+1,\ldots\}: \tau \in T\}$. 
Clearly $ht(T)\geq \alpha+n$ implies $ht(T^{(n)})\geq \alpha$ and in particular if $\lambda=ht(T)$ is a limit ordinal, then also $\lambda=ht(T^{(n)})$.

Note that for each $n\in \N$ there is some $m=m(n)\in \N$ such that if $|\tau|\geq m$, then there is a tree $T$ of height $n$ with $FS(T)\subseteq FS(\tau)$. 

We claim that if $A$ is an $\IP_{\lambda+m(n)}$-set (where $\lambda$ is a limit ordinal), then there exists a tree $T$ of height $\lambda + n$ such that $FS(T)\subseteq A$. Indeed pick a set $\tau$ with $|\tau|=m(n)$ such that $A'=\bigcap_{a\in FS(\tau)}A-a$ is an $\IP_\lambda$ set. 
Pick $\sigma$ such that $FS(\sigma)\subseteq FS(\tau)$ and the tree generated by $\sigma $ has height $n$. By IH there is a tree $T$ of height $\lambda$ with $FS(T)\subseteq A'$ and by the above comment we may additionally assume $T>\sigma$. It follows that $T+\sigma$ is a tree of height $\lambda+n$ with $FS(T+\sigma)\subseteq A$.
\end{proof}

\begin{definition}  If $r$ is an integer, an \emph{$r$-coloring} of a tree $T$ is a function $c:FS(T)\rightarrow [1,r]$.  We say $T$ is \emph{monochromatic} if $c$ is constant.
\end{definition}

Hindman's Theorem can be stated in the following form:
\begin{theorem}[Hindman's Theorem]
  If $c$ is an $r$-coloring of $T$ and $ht(T)=\infty$, there is a $T'\subseteq_{FS} T$ such that $ht(T')=\infty$ and $c$ is monochromatic on $FS(T')$.
\end{theorem}

\begin{theorem}
  The following are equivalent:
  \begin{enumerate}
  \item Hindman's Theorem.
  \item For every countable $\alpha$ and every $r$, there is a countable $\beta$ such that for any $T$ with $ht(T)\geq\beta$ and every $r$-coloring of $T$, there is a $T'\subseteq_{FS}T$ such that $ht(T)\geq\alpha$ and $c$ is constant on $FS(T')$.
  \end{enumerate}
\end{theorem}
\begin{proof}
  Suppose $(2)$, and let $c$ be a coloring of the integers.  Take the collection of all non-empty finite sets of integers to be a tree; this tree is clearly ill-founded, so $(2)$ implies that for every $\alpha$, we can find a $T^\alpha$ with $ht(T^\alpha)\geq\alpha$ and $c$ constant on $FS(T^\alpha)$.  But since there are only countably many finite sets of integers, we may choose an infinite set $\Lambda$ such that every finite subset of $\Lambda$ belongs to uncountably many $T^\alpha$, and therefore $c$ is constant on $FS(\Lambda)$.

  For the converse, note that for any $T$ and $c$, we can define $mono_i(T)$, the largest subtree of $T$ of color $i$, by $\sigma\in mono_i(T)$ iff $c$ is constantly equal to $i$  on $FS(\sigma)$ and there is a $\tau\in T$ such that $FS(\sigma)\subseteq FS(\tau)$.  Let $\alpha,r$ be given, and consider those $T$ and $r$-colorings $c$ of $T$ such that every monochromatic sub-tree of $T$ has height $\leq\alpha$.  Clearly all such trees are well-founded: if $T$ is ill-founded, it has an infinite branch, and that infinite branch has an infinite monochromatic subset, which in turn gives an infinite branch through $mono_i(T)$.  But the property of having a monochromatic sub-tree of height $\leq\alpha$ is $\Sigma^1_1$ (indeed, hyperarithmetic), and therefore by $\Sigma^1_1$-Bounding  (cf.\ for instance \cite[Section 35.D]{Kech95}), there must be a bound on the height of trees with this property, and adding $1$ to this bound gives the desired $\beta$.
\end{proof}

We conclude this section with an example which illustrates why the simple proof of Theorem \ref{SquaredFolkman} does not generalize in a straightforward way to higher ordinals. 
\begin{example}\label{NoCubedFolkman}
For any $\alpha$, there is a tree $T$ of height $\alpha\cdot\omega\cdot\alpha$ and a two-coloring of $T$ such that there exists no monochromatic tree of height larger then $\alpha\cdot\omega$. 

First choose a tree $T$ of height $\alpha\cdot\omega\cdot\alpha$ such that the following hold:
\begin{enumerate}
\item  The elements of $FS(T)$ are uniquely represented in the sense that  if $x\in FS(\sigma_1)$ and $x\in FS(\sigma_2)$ for some $\sigma_1, \sigma_2\in T$ then $x\in FS(\sigma_1\cap\sigma_2)$. 
\item If $k_1+\ldots +k_n\in FS(T)$ for some $k_i$ with each $\{k_i\}\in T$ then $\{k_1, \ldots, k_n\}\in T$. 
\item If $\tau<\sigma<\tau'$, $\tau\cup\sigma\in T$, $\sigma\cup\tau'\in T$, then $\tau\cup\sigma\cup\tau'\in T$.
\end{enumerate}
To see that such a tree exists, observe that it is simple to construct one in the semigroup of finite sets of integers (with $\cup$ as the operation and the collection of finite unions replacing $FS$), and that mapping a finite set $s$ to $\sum_{i\in s}2^i$ allows us to easily transfer this construction to our setting.

We may think of those $\{k\}\in T$ as ``basic elements'' of $T$; if $x\in FS(T)$ and $k_1^1+\cdots+k_n^1=k_1^2+\cdots+k_m^2=x$ then by the second property, $\{k_1^1,\ldots,k_n^1\},\{k_1^2,\ldots,k_m^2\}\in T$.  But then by the first property, $x\in FS(\{k_1^1,\ldots,k_n^1\}\cap\{k_1^2,\ldots,k_m^2\})$, so it must be that $m=n$ and $k_i^1=k_i^2$ for each $i\leq n$.

The third property implies that if $\sigma\in T$, $\tau<\sigma$, and $\tau\cup\sigma\in T$ then $ht(T-(\tau\cup\sigma))=ht(T-\sigma)$.

By the first property we may assign to each $x\in FS(T)$ the minimal (w.r.t.\ $\subseteq$) $\sigma_x \in T$ such that $x\in FS(T)$.

We first define $o(\sigma)$ for $\sigma\in T$ to be the unique ordinal $<\alpha$ such that $\alpha\cdot\omega\cdot o(\sigma)\leq ht(T-\sigma)<\alpha\cdot\omega\cdot (o(\sigma)+1)$.  For $x\in FS(T)$ with $\sigma_x=\{k_1<\cdots<k_n\}$, we define
\[c(x):=\left\{\begin{array}{ll}
1&\text{if }o(\{k_1\})=o(\{k_n\})\\
2&\text{if }o(\{k_1\})\neq o(\{k_n\})\\
\end{array}\right..\]
We must show that $ht(mono_1(T)),ht(mono_2(T))\leq\alpha\cdot\omega$.

To see that $ht(mono_1(T))\leq\alpha\cdot\omega$, first define $T^\gamma:=\{\sigma\mid \forall\tau\sqsubseteq\sigma o(\tau)=\gamma\}$ (that is, $T^\gamma$ consists of those $\{k_1<\cdots<k_n\}$ such that $o(\{k_1\})=o(\{k_n\})=\gamma$).  It is easy to see that $mono_1(T)=\bigcup_{\gamma<\alpha}T^\gamma$.  For any $k\in FS(mono_1(T))$, $mono_1(T)-\{k\}\subseteq T^\gamma$ for some $\gamma$.  Since clearly $ht(T^\gamma)=\alpha\cdot\omega$, it follows that $ht(mono_1(T)-\{k\})\leq\alpha\cdot\omega$, and therefore $ht(mono_1(T))\leq\alpha\cdot\omega$.

To see that $ht(mono_2(T))\leq\alpha\cdot\omega$, we will explicitly define a function $rk:mono_2(T)\rightarrow\alpha\cdot\omega$.  Given $\sigma\in mono_2(T)$, observe that $\sigma_\sigma:=\bigcup_{x\in\sigma}\sigma_x\in T$.  Let $o'(\sigma):=min_{x\in\sigma}o(\{\max\sigma_x\})=o(\max\{\sigma_\sigma\})$.  Let $n_\sigma$ be the largest integer such that there is a $\tau_\sigma$ with:
\begin{itemize}
\item $|\tau_\sigma|=n_\sigma$
\item $\tau_\sigma\cap\sigma_\sigma=\emptyset$
\item $\tau_\sigma\cup\sigma_\sigma\in T$
\item For each $k\in \tau_\sigma$, $o(\{k\})>o'(\sigma)$.
\end{itemize}
Define $rk(\sigma)=o'(\sigma)\cdot\omega+n_\sigma$.  Suppose $\tau\prec\sigma$ with $\tau,\sigma\in mono_2(T)$; if $o'(\tau)<o'(\sigma)$ then $rk(\tau)<rk(\sigma)$.  Otherwise, $o'(\tau)=o'(\sigma)$; consider $\tau':=\{k\in\sigma_{\tau\setminus\sigma}\mid o(k)>o'(\sigma)\}$.  $\tau'\neq\emptyset$ since $\tau'\in mono_2(T)$, and since $\sigma_{\tau\setminus\sigma}\cap\sigma_\sigma=\emptyset$, it follows that $n_\sigma\geq |\tau'|+n_\tau$, and in particular $rk(\tau)<rk(\sigma)$.

Since $rk(\sigma)<\alpha\cdot\omega$ for all $\sigma\in mono_2(T)$, it follows that $ht(mono_2(T))\leq\alpha\cdot\omega$.
\end{example}

In particular, setting $\alpha=\omega$, we see that there is a $2$-coloring of a tree of height $\omega^3$ whose monochromatic subsets have height at most $\omega^2$.

\section{A Proof of the Ordinally Effective Hindman's Theorem}\label{HindmansTheoremSection}
We give a proof ``unwinding'' the argument given in \cite{Tow09} to extract explicit information about ordinals from it.  A similar unwinding can be given for the proof of Hindman's Theorem due to Baumgartner, \cite{b74}, however because the reverse mathematical strength of that proof is higher (see \cite{blass87}), the ordinals bounds are much worse.  (In principle, an unwinding of Hindman's original proof should give ordinals bounds similar to the ones we give here, however that proof is sufficiently complicated that we are unsure what an unwinding would look like.)

To state the first lemma, it will be helpful to have the following {\em ad hoc} definition:
\begin{definition}
  $n\in FS_{\geq 2}(T)$ if there is a $\sigma\in T$ with $n=\sum_{i\in\sigma}i$ and $|\sigma|\geq 2$.
\end{definition}

\begin{lemma}
  Let $\sigma<\tau<T$, with $ht(T)\geq\alpha\cdot\beta$, and $c$ a coloring of $T+\tau+\sigma$ be given.  Then one of the following holds:
  \begin{itemize}
  \item There is a $T'\subseteq_{FS}T$, $ht(T')\geq\beta$ such that whenever $m\in FS_{\geq 2}(T'+\tau)$, there is an $n\in FS(\sigma)$ such that $c(m)=c(m+n)$.
  \item There is a $\tau'\in FS(T)$ and a $T'\subseteq_{FS}T-\tau'$ such that $ht(T')\geq\alpha$ and for every $n\in FS(T')$, there is an $m\in FS(\tau\cup\tau')$ such that for every $m'\in FS(\sigma)$, $c(n+m)\neq c(n+m+m')$.
  \end{itemize}
\end{lemma}
\begin{proof}
  By induction on $\beta$.  Let $\alpha,T,\sigma,\tau,c$ be given.  If $\beta=0$, the claim is trivial.  In the limit case $\beta=\sup_n\beta_n$, suppose the first case holds for every $n$.  That is, for each $n$, there is a $T'_n\subseteq_{FS}T$ with $ht(T'_n)\geq\beta_n$ and whenever $m\in FS_{\geq 2}(T'_n+\tau)$, there is an $m'\in FS(\sigma)$ such that $c(m)=c(m+m')$.  Then we may set $T':=\bigcup_nT'_n$, and $T'$ will satisfy the first case in the statement.  Otherwise the second case holds.

Suppose the claim holds for $\beta$ and $ht(T)\geq\alpha\cdot(\beta+1)=\alpha\cdot\beta+\alpha$.  Consider the tree $T_{\geq\alpha\cdot \beta}:=\{\sigma\in T\mid ht(T-\sigma)\geq\alpha\cdot\beta\}$.  Then $ht(T_{\geq\alpha\cdot\beta})\geq\alpha$.  If $\tau'=\emptyset, T'=T_{\geq\alpha\cdot\beta}$ satisfies the second condition, we are done.  Otherwise, there is a $k\in FS(T_{\geq\alpha\cdot\beta})$ such that for every $m\in FS(\tau)$ there is an $m'\in FS(\sigma)$ such that $c(k+m)=c(k+m+m')$.

Applying IH to $T-\{k\},\sigma,\tau\cup\{k\}$, if there is a $T'\subseteq_{FS}T-\{k\}$ with $ht(T')\geq\beta$ and such that whenever $m\in FS_{\geq 2}(T'+\tau\cup\{k\})$, there is an $n\in FS(\sigma)$ such that $c(m)=c(m+n)$, then $ht(T'+\{k\})\geq\beta+1$ and therefore satisfies the first condition.

On the other hand, if there is a $\tau'\in T-\{k\}$ and a $T'\subseteq_{FS}T-\{k\}-\tau'$ such that $ht(T')\geq\alpha$ and for every $n\in FS(T')$, there is an $m\in FS(\tau\cup\tau'\cup\{k\})$ such that for every $m'\in FS(\sigma)$, $c(n+m)\neq c(n+m+m')$, then $\tau'\cup\{k\}$ and $T'$ satisfy the second condition above.
\end{proof}

\begin{lemma}
  Let $\sigma<T$, with $ht(T)\geq\alpha\cdot 2\cdot\beta$, and $c$ a coloring of $T+\tau+\sigma$ be given.  Then one of the following holds:
  \begin{itemize}
  \item There is a $T'\subseteq_{FS}T$, $ht(T')\geq\beta$ such that whenever $m\in FS(T')$, there is an $n\in FS(\sigma)$ such that $c(m)=c(m+n)$.
  \item There is a $\tau'\in FS(T)$ and a $T'\subseteq_{FS}T-\tau'$ such that $ht(T')\geq\alpha$ and for every $n\in FS(T')$, there is an $m\in FS(\tau')$ such that for every $m'\in FS(\sigma)$, $c(n+m)\neq c(n+m+m')$.
  \end{itemize}
\end{lemma}
\begin{proof}
  We may apply the previous lemma to $\sigma,\emptyset,T,\alpha,2\cdot\beta$.  The second case is identical to the second case here.  In the first case, we obtain $T'\subseteq_{FS}T$ with $ht(T')\geq 2\cdot\beta$ such that whenever $m\in FS_{\geq 2}(T')$, there is an $n\in FS(\sigma)$ such that $c(m)=c(m+n)$.  Define $T''\subseteq_{FS}T'$ by $\sigma\in T''$ if there is a $\tau\in T'$, $\tau=\{n_1<n_2<\ldots<n_{2k-1}<n_{2k}\}$, with $\sigma=\{n_1+n_2,\ldots,n_{2k-1}+n_{2k}\}$.
\end{proof}

\begin{lemma}
  Let $ht(T)\geq (2\alpha)^{r+1}+1$ and $c$ an $r$-coloring of $T$.  There is a $\sigma\in T$ and a $T'\subseteq_{FS}T-\sigma$, $ht(T')\geq\alpha$, such that for every $n\in FS(T')$ there is an $m\in FS(\sigma)$ such that $c(n)=c(n+m)$.
\end{lemma}
\begin{proof}
  We apply the previous lemma repeatedly: let $T_0:=T$, let $\sigma_1=\{d\}$ where $ht(T-\{d\})\geq (2\alpha)^r$, and given $\sigma_{i+1}$, let $T_{i+1}:=T_i-\sigma_{i+1}$.  Given $\sigma_{i+1},T_{i+1}$ apply the previous lemma to $\bigcup_{j\leq i+1}\sigma_{j},T_{i+1},\alpha,\alpha^{r+1-i}$.  If the first case holds, we have a $T'$ with the desired property.

  Suppose we construct $\sigma_1,\ldots,\sigma_r$ with the second case in the previous lemma holding at each step.  Then we have $ht(T_{r+1})\geq\alpha$, and for each $n\in FS(T_{r+1})$, each $k< r$, and each $m\in FS(\bigcup_{j=k+1}^r\sigma_j)$, there is an $m'\in FS(\sigma_k)$ such that for every $m''\in FS(\bigcup_{j=1}^{k-1}\sigma_j)$, $c(n+m+m')\neq c(n+m+m'+m'')$.  In particular, we obtain a sequence $m_r,\ldots,m_1$ with each $m_i\in FS(\sigma_i)$ such that $c(n+\sum_{j=k}^rm_j)\neq c(n+\sum_{j=k'}^rm_j)$ whenever $k\neq k'$.  In particular, there must be some $k$ such that $c(n)=c(n+\sum_{j=k}^rm_j)$.  Since $\sum_{j=k}^rm_j\in FS(\bigcup_{j=1}^r\sigma_j)$, it follows that for every $n\in FS(T_{r+1})$, there is an $m\in FS(\bigcup_{j=1}^r\sigma_j)$ such that $c(n)=c(n+m)$.
\end{proof}

\begin{definition}
  If $\sigma<T$ and $c$ is a coloring of $FS(T+\sigma)$, we say $\sigma$ \emph{half-matches} $T$ if for every $n\in FS(T)$ there is an $m$ in $FS(\sigma)$ such that $c(n)=c(n+m)$.  We say $\sigma$ \emph{full-matches} $T$ if for every $n\in FS(T)$ there is an $m$ in $FS(\sigma)$ such that $c(n)=c(n+m)=c(m)$.

If $\sigma$ half-matches $T$, we define the induced coloring $c_\sigma$ by $c_\sigma(n)=(m,c(n))$ where $m\in FS(\sigma)$ is least such that $c(n)=c(n+m)$.  If $\sigma$ full-matches $T$, we define the induced coloring $c_{s,\sigma}$ by $c_{s,\sigma}(n)=(m,c(n))$ where $m\in FS(\sigma)$ is least such that $c(n)=c(n+m)=c(m)$.

  If $T$ is a tree, $\alpha$ is an ordinal, and $c$ is a coloring of $FS(T)$, we define inductively the $\beta$-\emph{half-matching height} with respect to $c$ and base $\beta$, $m_\beta(c)\mhyphen{}ht(T)$, by:
  \begin{itemize}
  \item $m_\beta(c)\mhyphen{}ht(T)\geq 0$ if $ht(T)\geq\beta$
  \item $m_\beta(c)\mhyphen{}ht(T)\geq\alpha+1$ if there is a $\sigma\in T$ and a $T'\subseteq_{FS}T-\sigma$ such that $\sigma$ half-matches $T'$ and $m_\beta(c_\sigma)\mhyphen{}ht(T')\geq\alpha$
  \item $m_\beta(c)\mhyphen{}ht(T)\geq\lambda$ iff $m_\beta(c)\mhyphen{}ht(T)\geq\alpha$ for all $\alpha<\lambda$.
  \end{itemize}

We define inductively the $\beta$-\emph{full-matching height} with respect to $c$, $fm(c)\mhyphen{}ht(T)$, by:
  \begin{itemize}
  \item $fm(c)\mhyphen{}ht(T)\geq 0$
  \item $fm(c)\mhyphen{}ht(T)\geq\alpha+1$ if there is a $\sigma\in T$ and a $T'\subseteq_{FS}T-\sigma$ such that $\sigma$ full-matches $T'$ and $fm(c_{s,\sigma})\mhyphen{}ht(T')\geq\alpha$
  \item $fm(c)\mhyphen{}ht(T)\geq\lambda$ iff $fm(c)\mhyphen{}ht(T)\geq\alpha$ for all $\alpha<\lambda$.
  \end{itemize}
\end{definition}

We now introduce the first of three {\em ad hoc} rapidly growing functions on ordinals.
\begin{definition}
  We define $\beta_\alpha$ inductively by:
  \begin{itemize}
  \item $\beta_0:=\max\{\beta,2\}$
  \item $\beta_{\alpha+1}:=(\beta_\alpha)^\omega$
  \item $\beta_\lambda:=\sup_{\alpha<\lambda}\beta_\alpha$.
  \end{itemize}
\end{definition}
The requirement that $\beta_0\geq 2$ is to prevent the definition from becoming degenerate.  For $\beta\geq 2$, $\beta_\alpha$ is the result of raising $\beta$ to the power $\omega$ iterated $\alpha$ times.  For example, for any $\beta<\epsilon_0$, $\beta_\omega=\epsilon_0$.  (The ordinals denoted $\epsilon_\delta$ are those ordinals greater than $1$ such that $\gamma<\epsilon_\delta$ implies $\omega^\gamma<\epsilon_\delta$, with $\epsilon_0$ being the first such ordinal.)  More generally, if $\gamma$ is the $\delta$-th limit ordinal and $\beta<\epsilon_0$ then $\beta_\gamma=\epsilon_\delta$.

\begin{lemma}
  Let $ht(T)\geq \beta_\alpha$ and $c$ an $r$-coloring of $T$.  There is a $T'\subseteq_{FS}T$ such that $m_\beta(c)\mhyphen{}ht(T')\geq\alpha$.
\end{lemma}
\begin{proof}
  By induction on $\alpha$.  If $\alpha=0$, $ht(T)\geq\beta$, and therefore $m_\beta(c)\mhyphen{}ht(T)\geq 0$.  When $\alpha$ is a limit, the claim follows immediately from IH.  Suppose the claim holds for $\alpha$ and $ht(T)\geq\beta_{\alpha+1}$.  Applying the preceding lemma to $T$, we obtain a $\sigma\in T$ and a $T'\subseteq_{FS}T-\sigma$ such that $ht(T')\geq\beta_\alpha$ and for every $n\in FS(T')$, there is an $m\in FS(\sigma)$ such that $c(n)=c(n+m)$.  We may apply IH to $T',c_\sigma$ to obtain a $T''\subseteq_{FS}T'$ such that $m_\beta(c_\sigma)\mhyphen{}ht(T'')\geq\alpha$.  Therefore $m_\beta(c)\mhyphen{}ht(T)\geq\alpha+1$.
\end{proof}

\begin{lemma}
  Let $T$ be a tree with $ht(T)\geq\alpha$, let $r$ be an integer, and let $\bigcup_{i\leq r}A_i=T$ where whenever $\sigma\in A_i$ and $\tau\subseteq\sigma$, $\tau\in A_i$.  Then there is a $T'\subseteq T$ such that $ht(T')\geq\alpha$ and some $i$ such that $T'\subseteq A_i$.
\label{downclose}
\end{lemma}
Note that we do not require the sets $A_i$ to be disjoint, and the downwards closure property means that it may be necessary for some $\tau$ to belong to both $A_i$ and $A_j$ (namely, if there are $\sigma_i,\sigma_j\supseteq\tau$ with $\sigma_i\in A_i$ and $\sigma_j\in A_j$).
\begin{proof}
  By induction on $\alpha$.  When $\alpha=0$ this is trivial, and when $\alpha$ is a limit, this follows immediately from IH.  Suppose the claim holds for $\alpha$ and $ht(T)\geq\alpha+1$.  Choose $\{n\}\in T$ such that $ht(T-\{n\})\geq\alpha$ and define $A'_i:=\{\sigma\in T-\{n\}\mid \sigma\cup\{n\}\in A_i\}$.  Clearly $T-\{n\}=\bigcup_{i\leq r}A'_i$, so by IH, there is a $T'\subseteq T-\{n\}$ such that $ht(T')\geq\alpha$ and there is an $i$ such that $T'\subseteq A'_i$.  But this, together with the downwards closure of $A_i$, implies that $T'+\{n\}\subseteq A_i$.
\end{proof}

\begin{lemma}
  Let $T>\sigma_0$ be a tree and $c,c'$ colorings.  Define $T^*\subseteq_{FS}T$ by setting $\tau\in T^*$ iff there is an $\{n\}\in T-\tau$ such that $c'(n)=c'(n+m)$ for all $m\in FS(\tau)$ and there is no $m\in FS(\tau\cup\sigma_0)$ such that $c(n)=c(n+m)= c(m)$.  If $m_\beta(c')\mhyphen{}ht(T)\geq\alpha$ then either:
  \begin{itemize}
  \item There is a $\sigma\in T$ and a $T'\subseteq_{FS}T-\sigma$ such that $ht(T')\geq\beta$ and $\sigma$ full-matches $T'$ with respect to $c'$, or
  \item $ht(T^*)\geq\alpha$.
  \end{itemize}
\end{lemma}
\begin{proof}
  By induction on $\alpha$.  When $\alpha=0$, the second condition holds trivially.  When $\alpha$ is a limit, either the second condition holds for all $\beta<\alpha$, and therefore for $\alpha$, or the first condition holds.  

  Suppose the claim holds for $\alpha$ and $m_\beta(c')\mhyphen{}ht(T)\geq\alpha+1$, and assume the first condition fails.  Choose $\sigma\in T$ so that $\sigma$ half-matches $T-\sigma$ and $m(c'_\sigma)\mhyphen{}ht(T-\sigma)\geq\alpha$.  Define $\tau\in T^*_0\subseteq_{FS}T-\sigma$ iff there is an $\{n\}\in T-\sigma-\tau$ such that $c'_\sigma(n)=c'_\sigma(n+m)$ for all $m\in FS(\tau)$ and there is no $m\in FS(\tau\cup\sigma\cup\sigma_0)$ with $c(n)=c(n+m)=c(m)$.  By IH, $ht(T^*_0)\geq\alpha$.

  For each $k\in FS(\sigma)$, let $A_k$ consist of those $\tau\in T^*_0$ such that there is some $n$ witnessing that $\tau$ belongs to $T^*_0$ such that $c'_\sigma(n)=(k,c'(n))$.  Clearly every $\tau\in T^*_0$ belongs to some $A_k$ and the $A_k$ are downwards closed, so by Lemma \ref{downclose} there must be a $T'\subseteq T^*_0$ such that $T'\subseteq A_k$ for some $k$ and $ht(T')\geq\alpha$.  For each $\tau\in T'$, since $c'_\sigma(n)=c'_\sigma(n+m)$ for each $m\in FS(\tau)$, it follows that $c'(n)=c'(n+m)$ for each $m\in FS(\tau\cup\{k\})$.  Since there is no $m\in FS(\tau\cup\sigma\cup\sigma_0)$ with $c(n)=c(n+m)=c(m)$, in particular there is no $m\in FS(\tau\cup\{k\}\cup\sigma_0)$ with $c(n)=c(n+m)=c(m)$.  Therefore $T'+\{k\}\subseteq T^*$, so $ht(T^*)\geq ht(T'+\{k\})\geq\alpha+1$.
\end{proof}

\begin{lemma}
  If $ht(T)\geq\beta_\alpha$ and $c$ is an $r$-coloring of $T$, either:
  \begin{itemize}
  \item There is a $\sigma\in T$ and a $T'\subseteq_{FS}T-\sigma$ such that $ht(T')\geq\beta$ and $\sigma$ full-matches $T'$, or
  \item There is a $T'\subseteq_{FS}T$ and an $i\leq r$ such that $ht(T')\geq\alpha$ and $c(m)\neq i$ for any $m\in FS(T')$.
  \end{itemize}
\label{splitty}
\end{lemma}
\begin{proof}
  Given $T$, we find $T'\subseteq_{FS}T$ such that $m_\beta(c)\mhyphen{}ht(T')\geq\alpha$.  Applying the previous lemma with $\sigma_0=\emptyset$ and $c=c'$, if the first condition holds, we are done.  In the second case we obtain $T^*$ with $ht(T^*)\geq\alpha$ and for each $i\leq r$ let $A_i$ consist of those $\tau\in T^*$ such that there is some $\{n\}\in T'-\tau$ witnessing that $\tau\in T^*$ with $c(n)=i$.  Since the $A_i$ are downwards closed, there is a $T''\subseteq T^*$ such that $T''\subseteq A_i$ for some $i$ and $ht(T'')\geq\alpha$.

  For each $\tau\in T''$, there is an $\{n\}\in T'-\tau$ such that $i=c(n)=c(n+m)$ for each $m\in FS(\tau)$, and therefore $c(m)\neq i$.
\end{proof}

We introduce our second rapidly growing function on ordinals:
\begin{definition}
  \begin{itemize}
  \item $\beta_{\alpha,0}:=\beta$
  \item $\beta_{\alpha,n+1}:=\beta_{\beta_{\alpha,n}}$.
  \end{itemize}
\end{definition}
This function grows very quickly; for instance, $\omega_{\omega,2}=\omega_{\omega_{\omega,1}}=\omega_{\epsilon_0}=\epsilon_{\epsilon_0}$.


\begin{lemma}
  If $ht(T)\geq (\max\{\alpha,\beta\})_{\alpha\cdot(r-1)}$ and $c$ is an $r$-coloring of $T$, either:
  \begin{itemize}
  \item There is a $\sigma\in T$ and a $T'\subseteq_{FS}T-\sigma$ such that $ht(T')\geq\beta$ and $\sigma$ full-matches $T'$, or
  \item There is a $T'\subseteq_{FS}T$ such that $ht(T')\geq\alpha$ and $c$ is constant on $T'$.
  \end{itemize}
\end{lemma}
\begin{proof}
  By induction on $r$.  When $r=1$, the second condition holds trivially.  If the claim holds for $r$ and $ht(T)\geq (\max\{\alpha,\beta\})_{\alpha\cdot r}=\left((\max\{\alpha,\beta\})_{\alpha\cdot(r-1)}\right)_{\alpha}$ then, by Lemma \ref{splitty}, either the first condition holds or there is a $T'\subseteq_{FS}T$ and an $i\leq r$ such that $ht(T')\geq(\max\{\alpha,\beta\})_{\alpha,r-1}$ and $c(m)\neq i$ for any $m\in FS(T')$.  We may reindex $c$ so that it is an $r-1$-coloring on $T'$, and then the claim follows by IH.
\end{proof}

The following is the final rapidly growing function on ordinals we need, since it will describe the final bound we obtain:
\begin{definition}
  \begin{itemize}
  \item $f(\alpha,0):=0$
  \item $f(\alpha,\beta+1):=\sup_n(f(\alpha,\beta))_{\alpha,n}$
  \item $f(\alpha,\lambda):=\sup_{\beta<\lambda}f(\alpha,\beta)$.
  \end{itemize}
\end{definition}

\begin{lemma}
  If $ht(T)\geq f(\alpha,\beta)$ and $c$ is an $r$-coloring of $T$ then either:
  \begin{itemize}
  \item There is a $T'\subseteq_{FS}T$ such that $fm(c)\mhyphen{}ht(T')\geq\beta$, or
  \item There is a $T'\subseteq_{FS}T$ such that $ht(T')\geq\alpha$ and $c$ is constant on $T'$.
  \end{itemize}
\label{final1}
\end{lemma}
\begin{proof}
  By induction on $\beta$.  When $\beta=0$ the claim is trivial, and when $\beta$ is a limit, the claim follows immediately from IH.  So suppose the claim holds for $\beta$ and $ht(T)\geq f(\alpha,\beta+1)$.  By the previous lemma, either there is a $T'$ witnessing the second condition or there is a $\sigma\in T$ and a $T'\subseteq_{FS}T-\sigma$ such that $ht(T')\geq f(\alpha,\beta)$ and $\sigma$ full-matches $T'$.  In the latter case, IH applied to $T'$ and $c_{s,\sigma}$ gives either the second condition or a $T''$ such that $fm(c_{s,\sigma})\mhyphen{}ht(T'')\geq\beta$, and therefore $fm(c)\mhyphen{}ht(T''+\sigma)\geq\beta+1$.
\end{proof}

\begin{lemma}
  If $fm(c)\mhyphen{}ht(T)\geq\alpha$ then there is a $T'\subseteq_{FS}T$ such that $ht(T')\geq\alpha$ and for each $\tau\in T'$, $c$ is constant on $FS(\tau)$.
\end{lemma}
\begin{proof}
  By induction on $\alpha$.  When $\alpha=0$ the claim is trivial, and when $\alpha$ is a limit, the claim follows immediately from IH.  So suppose the claim holds for $\alpha$ and $fm(c)\mhyphen{}ht(T)\geq\alpha+1$.  There is a $\sigma\in T$ such that $\sigma$ full-matches $T-\sigma$ and $fm(c_{s,\sigma})\mhyphen{}ht(T-\sigma)\geq\alpha$.  By IH, there is a $T'\subseteq_{FS}T-\sigma$ such that $ht(T')\geq\alpha$ and for each $\tau\in T'$, $FS(\tau)$ is monochromatic under $c_{s,\sigma}$.

Define $A_k$ to be those $\tau\in T'$ such that $c_{s,\sigma}(m)=(k,c(m))$ for each $m\in FS(\tau)$.  Then there is a $T''\subseteq T'$ with $ht(T'')\geq\alpha$ and $T''\subseteq A_k$ for some $k$.  Then $T''+\{k\}$ satisfies the claim.
\end{proof}

\begin{lemma}
  If $fm(c)\mhyphen{}ht(T)\geq\alpha$ then there is a $T'\subseteq_{FS}T$ such that $ht(T')\geq\alpha$ and $T$ is monochromatic under $c$.
\label{final2}
\end{lemma}
\begin{proof}
  Apply the previous lemma to obtain $T'\subseteq_{FS}T$ such that $ht(T')\geq\alpha$ and for each $\tau\in T'$, $c$ is constant on $FS(\tau)$.  For $i\leq r$, define $A_i$ to be those $\tau\in T'$ such that $c$ is constantly equal to $i$ on $FS(\tau)$.  There is a $T''\subseteq T'$ such that $T''\subseteq A_i$ for some $i$ and $ht(T'')$; this $T''$ satisfies the claim.
\end{proof}

\begin{theorem}\label{OrdinalEffectiveHindman}
  If $ht(T)\geq f(\alpha,\alpha)$ and $c$ is an $r$-coloring of $T$ then there is a $T'\subseteq_{FS}T$ such that $ht(T')\geq\alpha$ and $c$ is constant on $T'$.
\end{theorem}
\begin{proof}
  By Lemma \ref{final1}, either the conclusion holds, or we may apply Lemma \ref{final2} to obtain the claim.
\end{proof}

\section{Further Developments}
Most Ramsey theoretic properties giving infinite sets will be approximated by transfinite constructions in a similar way.  For example, in the case of Ramsey's Theorem for pairs, we could define a set $A\subseteq[\mathbb{N}]^2$ to be $1$-Ramsey if it is non-empty, $\alpha+1$-Ramsey if there is an $n$ such that $\{\{m,m'\}\in A\mid \{n,m\}\in A \wedge \{n,m'\}\in A\}$ is $\alpha$-Ramsey, and $A$ is $\lambda$-Ramsey if for every $\beta<\lambda$, $A$ is $\beta$-Ramsey.  Then Ramsey's Theorem for pairs would be equivalent to the statement that for every $c$ and every countable $\alpha$, there exists a monochromatic $\alpha$-Ramsey set.  A closely related family of approximations for the pigeonhole principle has been studied by Tao \cite{Tao07} and Gaspar and Kohlenbach \cite{GK10}.

We hope that these approximations might give tractable fragments of open problems.  For instance, the following question is open:
\begin{question}
  Let $k\geq 2$.  Is there a set $D\subseteq\mathbb{N}$ such that whenever $D$ is finitely colored, there is an infinite $D'\subseteq D$ such that all sums of $\leq k$ elements of $D'$ belong to the same color, but such that $D$ is not an  $\IP_{k+1}$-set?
\end{question}
The finite version was shown in \cite{NR90}, but as far as we know, even the simplest transfinite analogs are open.  For instance, consider the following, the $\omega+1$ case of a weaker question asked in \cite{HLS03}:
\begin{question}
  Is there a set $D\subseteq\mathbb{N}$ which is not $\IP_3$, but such that whenever $D$ is finitely colored, there is a $d\in D$ and for every $n$, $b^n_1,\ldots,b^n_n\in D$, such that $d$, $b^n_i$ for all $n,i\leq n$, and all sums of the form $d+b^n_i$ or $b^n_i+b^n_j$ ($i\neq j$) are the same color?
\end{question}

A similar question is:
\begin{question}
  Is there a set $D\subseteq\mathbb{N}$ such that whenever $D$ is finitely colored, there is a $d\not\in D$ and for every $n$, a sequence $b^n_1,\ldots,b^n_n\not\in D$ such that $FS(d,b^n_1,\ldots,b^n_n)\setminus\{d,b^n_1,\ldots,b^n_n\}$ is monochromatic (and contained in $D$) for every $n$?
\end{question}
This is the $\omega+1$ version of a statement whose finite version is proven in \cite{Pro85} (with a simpler proof given in \cite{DHLS00}).

Finally, we note that the ordinal bounds given here are not necessarily optimal.  Indeed,  true ordinal bounds are closely related to reverse mathematical strength, which remains open for Hindman's Theorem.  The bounds given here are consistent with the upper bound on the strength of Hindman's Theorem given in \cite{blass87}.


\def\ocirc#1{\ifmmode\setbox0=\hbox{$#1$}\dimen0=\ht0 \advance\dimen0
  by1pt\rlap{\hbox to\wd0{\hss\raise\dimen0
  \hbox{\hskip.2em$\scriptscriptstyle\circ$}\hss}}#1\else {\accent"17 #1}\fi}

\end{document}